\documentclass[12pt,a4paper]{amsart}
\usepackage{amssymb,enumerate}

\usepackage[pagebackref]{hyperref}

\newcommand{\CC}{{\mathbb{C}}}
\newcommand{\FF}{{\mathbb{F}}}
\newcommand{\RR}{{\mathbb{R}}}

\newcommand{\bB} {\mathbf B}
\newcommand{\bC} {\mathbf C}
\newcommand{\bE} {\mathbf E}
\newcommand{\bF} {\mathbf F}
\newcommand{\bG} {\mathbf G}
\newcommand{\bO}{{\mathbf O}}
\newcommand{\bN} {\mathbf N}
\newcommand{\bT} {\mathbf T}
\newcommand{\bU} {\mathbf U}
\newcommand{\bZ} {\mathbf Z}

\newcommand{\cE} {\mathcal E}

\newcommand{\fA} {\mathfrak A}
\newcommand{\fP} {\mathfrak{P}}
\newcommand{\fS} {\mathfrak S}

\newcommand{\Aut}{{\operatorname{Aut}}}
\newcommand{\cd}{\operatorname{cd}}
\newcommand{\diag}{{\operatorname{diag}}}
\newcommand{\Irr}{{\operatorname{Irr}}}
\newcommand{\mh}{{{\operatorname{mh}}}}
\newcommand{\Out}{{\operatorname{Out}}}
\newcommand{\Stab}{{\operatorname{Stab}}}
\newcommand{\Syl}{{\operatorname{Syl}}}

\newcommand{\GL}{\operatorname{GL}}
\newcommand{\AGL}{\operatorname{AGL}}
\newcommand{\GaL}{\operatorname{\Gamma L}}
\newcommand{\AGaL}{\operatorname{A\Gamma L}}
\newcommand{\SL}{\operatorname{SL}}
\newcommand{\PGL}{\operatorname{PGL}}
\newcommand{\PSL}{\operatorname{PSL}}
\newcommand{\PSU}{\operatorname{PSU}}
\newcommand{\Sp}{{{\operatorname{Sp}}}}
\newcommand{\SO}{{{\operatorname{SO}}}}

\newcommand{\tw}[1]{{}^{#1}\!}
\newcommand{\wal}{{\widehat{\alpha}}}

\let\al=\alpha
\let\be=\beta
\let\eps=\epsilon
\let\ga=\gamma
\let\la=\lambda
\let\Om=\Omega
\let\om=\omega
\let\vhi=\varphi

\newtheorem{thm}{Theorem}[section]
\newtheorem{lem}[thm]{Lemma}
\newtheorem{cor}[thm]{Corollary}
\newtheorem{prop}[thm]{Proposition}

\newtheorem{thmA}{Theorem}

\theoremstyle{definition}

\numberwithin{equation}{section}

\begin{document}

\title[Minimal heights and defect groups with two character degrees]{Minimal heights and defect groups\\ with two character degrees}

\author{Gunter Malle, Alexander Moret\'o and Noelia Rizo}
\address{FB Mathematik, TU Kaiserslautern, Postfach 3049,
  67653 Kaisers\-lautern, Germany.}
\email{malle@mathematik.uni-kl.de}

\address{Departament de Matem\`atiques, Universitat de Val\`encia, 46100
 Burjassot, Val\`encia, Spain}
\email{alexander.moreto@uv.es}
\email{Noelia.Rizo@uv.es}

\begin{abstract}
Conjecture A of \cite{EM14} predicts the equality between the smallest positive
height of the irreducible characters in a $p$-block of a finite group and the
smallest positive height of the irreducible characters in its defect group. 
Hence, it can be seen as a generalization of Brauer's famous height zero
conjecture. One inequality was shown to be a consequence of Dade's Projective
Conjecture. We prove the other, less well understood, inequality for principal
blocks when the defect group has two character degrees. 
\end{abstract}

\thanks{
The first author gratefully acknowledges financial support by SFB TRR
195 -- Project-ID 286237555.
The second and third authors are supported by Ministerio de Ciencia e
Innovaci\'on (Grants PID2019-103854GB-I00 and PID2022-137612NB-I00 funded by
MCIN/AEI/ 10.13039/501100011033 and ERDF A way of making Europe) and
Generalitat Valenciana Grants (CIAICO/2021/163 and CDEIGENT grant
CIDEIG/2022/29).} 

\keywords{Blocks, defects, character heights}

\subjclass[2010]{Primary 20C15, 20C20, 20C33}

\date{\today}

\maketitle


\section{Introduction}

One of the main problems of the past decades in representation theory of finite
groups is Brauer's Height Zero Conjecture. Posed by Richard Brauer in 1955, it
asserts that all irreducible characters in a Brauer $p$-block $B$ of a finite
group $G$ have height zero if and only if the defect groups of $B$ are
abelian. The proof of this conjecture has recently been completed in
\cite{MNST} (building among others on earlier work in \cite{NT13, KM, Ru}).
Now it seems appropriate to study possible extensions to blocks with
non-abelian defect groups.

One such possibility was proposed in Conjecture A of \cite{EM14}, where it was
conjectured that if $p$ is a prime, $B$ is a Brauer $p$-block with non-abelian
defect group $D$ and $\Irr(B)$ is the set of irreducible characters in $B$,
then the smallest positive height of the non-linear irreducible characters in
$D$, i.e.,
$$\mh(D)=\min\{\log_p \vhi(1)\mid\vhi\in\Irr(D), \vhi(1)>1\},$$
coincides with the smallest positive height of the non-linear characters in
$\Irr(B)$, which we denote $\mh(B)$. In the following, we will refer to this
conjecture as the EM-conjecture. If we adopt the convention that $\mh(B)$ or
$\mh(D)$ are infinity if no such characters of positive height exist, the
EM-conjecture  extends Brauer's Height Zero Conjecture. The next natural case
to consider is when the defect group has two character degrees. 

In \cite{EM14} it was proved that the inequality $\mh(D)\leq\mh(B)$ follows from
Dade's Projective Conjecture, which has been reduced to a problem on simple
groups. There is still less evidence for the inequality $\mh(D)\geq\mh(B)$. We
refer the reader to \cite{EM14,BM15,flz,sam} for some cases where the conjecture
is known to hold and to \cite[Conj.~2.9]{klls} for a reformulation in terms of
fusion systems.

Our investigations give strong evidence for the validity of the conjecture when
$D$ has two character degrees. As usual, given a finite group $G$ and a fixed
prime~$p$, we write $B_0(G)$ to denote the principal $p$-block of $G$. Our main
result is the following.

\begin{thmA}
 Let $p$ be a prime and let $G$ be a finite group. If a Sylow $p$-subgroup $P$
 of $G$ has two character degrees, then $\mh(B_0(G))\leq \mh(P)$. 
\end{thmA}

In other words, we prove the less well understood inequality of the
EM-conjecture for principal $p$-blocks of groups whose Sylow $p$-subgroups have
two character degrees (the case with just one degree being Brauer's Height Zero
Conjecture). Our proof depends on the Classification of Finite Simple Groups
and the Deligne--Lusztig theory of characters of finite reductive groups.
Another fundamental tool that we use is the work in \cite{GLPST}. It may be
worth remarking that the recent solution \cite{MNST} of Brauer's Height Zero
Conjecture (BHZ) was preceded by the 
easier proof in the case of principal blocks in \cite{mn}. While the proof of
the general case of BHZ also depends heavily on \cite{GLPST} by means of
\cite{NT13}, the proof of BHZ for principal blocks does not. We think that
extending our main result to arbitrary blocks will be a very challenging
problem.
\medskip

Our paper is structured as follows: in Section~\ref{sec:linear} we collect some
results on linear and permutation groups that are needed in the proof of
Theorem 1. In Section~\ref{sec:simple} we prove some results
on simple and almost simple groups, in particular we show a stronger version of
Theorem 1 for almost simple groups that may be useful elsewhere (see
Theorem~\ref{thm:almost}). In Section~\ref{sec:main} we prove Theorem~1.
\medskip

\noindent{\bf Acknowledgements:}
We thank the referees for their careful reading and helpful suggestions, and
Mandi Schaeffer Fry for making us aware of an omission in one of our proofs.

\section{Linear and permutation groups}   \label{sec:linear}

Let $p$ be a prime. In the following, $G$ will always be a finite group.
Following \cite{GLPST}, we say that a linear group $G\leq\GL_n(p)$ is
\emph{$p$-exceptional} if $p$ divides $|G|$ and all $G$-orbits on the natural
module have $p'$-size, that is, their length is not divisible by $p$.
The classification of solvable
irreducible $p$-exceptional groups was already an impressive result (see
Sections~9 and~10 of \cite{MW}), which lies at the heart of the proof of the
$p$-solvable case of BHZ by D.~Gluck and T.~R.~Wolf \cite{gw}. It took more
than $30$ years and a team of experts \cite{GLPST} to extend this result to
arbitrary groups.  Since it is fundamental for our work, we recall it here.

\begin{thm}   \label{t1}
 Let $G$ be an irreducible $p$-exceptional subgroup of $\GL_n(p)=\GL(V)$
 and suppose that $G$ acts primitively on $V$. Then one of the following holds:
 \begin{enumerate}[\rm(1)]
  \item $G$ is transitive on $V\setminus\{0\}$;
  \item $G\leq\GaL_1(p^n)$;
  \item $G$ is one of the following:
   \begin{enumerate}[\rm(i)]
    \item $G=\fA_c, \fS_c$ where $c=2^r-2$ or $2^r-1$, with $V$ the deleted
     permutation module over $\FF_2$, of dimension $c-2$ or $c-1$
     respectively;
    \item $\SL_2(5)\trianglelefteq G< \GaL_2(9)<\GL_4(3)$,
     orbit sizes $1,40,40$;
    \item $\PSL_2(11)\trianglelefteq G<\GL_5(3)$, orbit sizes $1,22,110,110$;
    \item $M_{11}\trianglelefteq G<\GL_5(3)$, orbit sizes $1,22,220$;
    \item $M_{23}=G<\GL_{11}(2)$, orbit sizes $1,23,253,1771$.
   \end{enumerate}
 \end{enumerate}
\end{thm}

\begin{proof}
This is Theorem 1 of \cite{GLPST}.
\end{proof}

The classification of the groups in case (1) was achieved by Hering in another
deep theorem (see \cite{He74}).  We will use the description of these groups
that appears in Appendix~1 of \cite{liebeck}. For the imprimitive case of
Theorem~\ref{t1}, a result on permutation groups is necessary. Given a
prime~$p$, a subgroup $G\leq\fS_n$ is \emph{$p$-concealed} if it has order
divisible by~$p$ and all the orbits on the power set of $\{1,\dots,n\}$ have
$p'$-size.

\begin{thm}   \label{t2}
 Let $H$ be a primitive subgroup of $\fS_n$ of order divisible by a prime $p$.
 Then $H$ is $p$-concealed if and only if one of the following holds:
 \begin{enumerate}[\rm(1)]
  \item $\fA_n\trianglelefteq H\leq \fS_n$, and $n=ap^s-1$ with $s\geq 1$,
   $a\leq p-1$ and $(a,s)\neq(1,1)$; also $H\neq\fA_3$ if $(n,p)=(3,2)$;
  \item $(n,p)=(8,3)$, and $H=\AGL_3(2)=2^3:\SL_3(2)$ or
   $H=\AGaL_1(8)=2^3:7:3$;
  \item $(n,p)=(5,2)$ and $H=D_{10}$. 
 \end{enumerate}
 In cases $(2)$ and $(3)$ there exists $\chi\in\Irr(H)$ such that
 $\chi(1)_p=p$. 
\end{thm}

\begin{proof}
The first part is \cite[Thm~2]{GLPST}. The final assertion can be easily
checked.
\end{proof}

Given a finite group $G$, we write $\bO^{p'}(G)$ to denote the smallest normal
subgroup of $G$ with $p'$ factor group, similarly $\bO^p(G)$. Analogously, we
write $\bO_{p'}(G)$ (resp.~$\bO_{p}(G)$) to denote the largest normal
$p'$-subgroup (resp. $p$-subgroup) of $G$.

\begin{thm}   \label{t3}
 Suppose $G\leq \GL_n(p)=\GL(V)$ is irreducible and $p$-exceptional with
 $G=\bO^{p'}(G)$. If $V=V_1\oplus\cdots\oplus V_n$ ($n>1$) is an
 imprimitivity decomposition for $G$, then $G_{V_1}$ is transitive on
 $V_1\setminus\{0\}$ and $G$ induces a primitive $p$-concealed subgroup of
 $\fS_\Om$, where $\Om=\{V_1,\dots, V_n\}$.
\end{thm}

\begin{proof}
This is \cite[Thm~3]{GLPST}.
\end{proof}

We will use the following elementary result. In this lemma, the word block will
have a different meaning than in the rest of the paper. Suppose that a group
$G$ acts transitively on a set $\Om$, and let $\Delta\subseteq \Om$.
Recall that $\Delta$ is a \emph{block} if for all $g\in G$, either
$\Delta^g\cap \Delta=\emptyset$ or $\Delta^g=\Delta$.

\begin{lem}   \label{pconcealed}
 Suppose that $G$ acts transitively but not primitively on
 $\Om=\{1,\ldots, n\}$. Let $\Delta\subsetneq \Om=\{1,\ldots, n\}$ be a
 block of maximal order. Then:
 \begin{enumerate}[\rm(a)]
  \item $\Om=\Delta_1\cup\cdots\cup \Delta_r$, where $\Delta_i$ are the
   translates of $\Delta$. This union is disjoint and $|\Delta_i|=|\Delta|$
   for all $i$.
  \item Let $L=\bigcap_{i=1}^r\Stab_G(\Delta_i)$. Then the induced action
   of $G/L$ on $\{\Delta_1,\ldots,\Delta_r\}$ is primitive.
  \item If the action of $G$ on $\Om$ is $p$-concealed, the action of $G/L$
   on $\{\Delta_1,\ldots,\Delta_r\}$ is $p$-concealed.
 \end{enumerate}
\end{lem}

\begin{proof}
Part (a) and (b) are well-known.

For the last part, write $\bar{G}=G/L$. Suppose that
$\Gamma=\{\Delta_1,\ldots,\Delta_s\}$ is a block for the action of $\bar{G}$
and write $\Gamma'=\bigcup_{i=1}^s\Delta_i$. Notice that
we have $\Stab_{\bar{G}}(\Gamma)=\Stab_G(\Gamma')/L$. Now,
$$|\bar{G}:\Stab_{\bar{G}}(\Gamma)|=|G/L:\Stab_G(\Gamma')/L|
  =|G:\Stab_G(\Gamma')|$$
is not divisible by $p$.
\end{proof}

For a set $\Om$, we write $\fP(\Om)$ to denote the power set of
$\Om$. More generally, following \cite{Do}, for any integer $n>1$ we set
$$\fP_n(\Om)=\{(\Lambda_1,\dots,\Lambda_n)\mid\Om=\bigcup_{i=1}^n\Lambda_i,\ 
  \Lambda_i\cap\Lambda_j=\emptyset\text{\ for $i\neq j$}\}.
$$
By a theorem of Gluck, see \cite{Gl} or \cite[Thm~5.6]{MW}, an odd order
permutation group on $\Om$ has a regular orbit on $\fP(\Om)$ or, equivalently,
on $\fP_2(\Om)$ (see for instance \cite[Cor.~5.7(b)]{MW}).
As pointed out in \cite[Prop.~1]{Gl}, this is false for $2$-groups.
In \cite[Cor.~6]{Do} and independently in \cite[Thm~1.2]{Se} it was proved that
a solvable permutation group on $\Om$ has a regular orbit on $\fP_4(\Om)$.
We will need a regular orbit on $\fP_3(\Om)$ in the case when the group is
a $p$-group. (Again, as examples in \cite{Do} and \cite{Se} show, this is not
true for arbitrary solvable groups.)
By the previous comments, it suffices to do it when $p=2$.

\begin{lem}   \label{dol}
 Let $P\leq\fS_\Om$ be a $p$-group.  If $p$ is odd, then there exists a
 regular $P$-orbit on $\fP(\Om)$. If $p=2$ then there exists a regular
 $P$-orbit on $\fP_3(\Om)$.
\end{lem}

\begin{proof}
The first part follows from \cite[Cor.~5.7(b)]{MW}.
The second part is a consequence of work in \cite{Do}. Indeed, by
\cite[Lemma~1(c)]{Do}, every proper primitive group $H<\fS_\Om$ has at
least four regular orbits on $\fP_3(\Om)$. It follows from \cite[Thm~2]{Do}
that $P$ has a regular orbit on $\fP_3(\Om)$, as desired.
\end{proof}

\section{Results on almost simple groups}   \label{sec:simple}
In this section we collect some results on almost simple groups that will be
used in the proof of Theorem 1.

\subsection{Miscellaneous results}
We start with an elementary result on the groups $\PSL_2$. 

\begin{lem}   \label{lemasl}
 Let $S\cong\SL_2(p^n)$ or $S\cong\PSL_2(p^n)$, with $(p,n)\neq (2,1),(3,1)$.
 Let $R\in\Syl_p(S)$. Then:
 \begin{enumerate}[\rm(a)]
  \item If $R\leq H<S$, then $H\leq \bN_S(R)$.
  \item If $Q\in\Syl_p(S)$ with $Q\neq R$, then $Q\cap R=1$.
 \end{enumerate} 
\end{lem}

\begin{proof}
Suppose first that $S\cong \PSL_2(p^n)$. Part (a) follows from Dickson's
classification of the subgroups of $\PSL_2(p^n)$ (see
\cite[Hauptsatz~II.8.27]{hup}) and part~(b) follows from
\cite[Satz~II.8.2]{hup}. Now, suppose that $S=\SL_2(p^n)$ and let $Z=\bZ(S)$
so that $S/Z$ is as in the previous paragraph. Let $H<S$ containing
$R\in\Syl_p(S)$. Therefore, $HZ/Z\leq \bN_{S/Z}(RZ/Z)$. Since $Z$ is a
$p'$-group, $R$ is normal in $H$, as wanted. Claim (b) also follows from the
previous case. 
\end{proof}

Next, we obtain a result on the character degrees of alternating groups.

\begin{lem}   \label{alt}
 Let $p$ be an odd prime and $p\leq n<p^2$, $n>4$ and $(n,p)\neq (6,3)$. Then
 there exists $\chi\in\Irr(\fA_n)$ such that $\chi(1)_p=p$.
\end{lem}

\begin{proof}
If $p=3$, the result can be easily checked. So we may assume that $p>3$. Write
$n=ap+b$ with $1\leq a<p$ and $0\leq b<p$. By the hook length formula
\cite[Thm~2.3.21]{JK}, given a partition $\la$ of $n$, the character
$\chi^\la\in\Irr(\fS_n)$ indexed by $\la$ has degree
$$\chi^\la(1)=\frac{n!}{\prod_{i,j}h_{ij}^{\la}},$$
where $h_{ij}^\lambda$ are the hooklengths of $\la$. Let $\chi\in\Irr(\fA_n)$
under $\chi^\la$. Since $p$ is odd, we have $\chi(1)_p=\chi^\la(1)_p$, so
it is enough to find a character $\chi^\la\in\Irr(\fS_n)$ such that
$\chi^\la(1)_p=p$. Since $(n!)_p=p^a$, we need to find a partition $\la$ such
that $\prod_{i,j}h_{ij}^{\la}=p^{a-1}$. If $b\neq 0$, we take the partition
$\la=(ap,1,1,\ldots,1)$. If $b=0$, we take the partition $\la=(ap-2,2)$.
\end{proof}

We need the following character extendability result in almost simple groups.
As usual, given a group $G$ and a prime $p$, $\Irr_{p'}(B_0(G))$ is the set of
$p'$-degree irreducible characters in the principal $p$-block of $G$.

\begin{lem}   \label{simples}
 Let $p$ be a prime and let $S$ be a non-abelian simple group. Then there exists
 $1_S\neq \alpha\in\Irr_{p'}(B_0(S))$ which is $X$-invariant for some Sylow
 $p$-subgroup $X\in\Syl_{p}(\Aut(S))$.
\end{lem}

\begin{proof}
This follows from \cite[Prop.~2.1]{GRSS} when $p>3$ and from \cite[Thm~C]{RSV}
when $p\leq3$.
\end{proof}

Finally, we prove a technical lemma that will be used several times in the
proof of Theorem 1.

\begin{lem}   \label{pcon}
 Let $p$ be a prime. Let $H$ be a primitive subgroup of $\fS_n$ of order
 divisible by $p$ with abelian Sylow $p$-subgroups such that $\bO^{p'}(H)=H$.
 If $H$ is $p$-concealed, then there exists $\chi\in\Irr(H)$ such that
 $\chi(1)_p=p$.
\end{lem}

\begin{proof}
By hypothesis, $H$ is one of the groups listed in Theorem~\ref{t2}. We have
already seen in Theorem~\ref{t2} that in cases (2) and (3) the result holds.
Now, assume that we are in case (1) of that theorem, so
$\fA_r\lhd H \leq \fS_r$, with $r=ap^s-1$, $s\geq 1$ and $(a,s)\neq (1,1)$. If
$r=2$, we have $p=3$ and $(a,s)=(1,1)$, a contradiction. If $r=3$ then $p=2$
and $H\neq \fA_3$ by hypothesis, so $H\cong \fS_3$, and there is
$\chi\in\Irr(H)$ of degree $2$. We are done in this case. If $r=4$, we have
$p=5$ and $(a,s)=(1,1)$, another contradiction. If $r=6$, then $p=7$ and
$(a,s)=(1,1)$, contradiction. Hence we have that $r\geq 5$, $r\neq 6$.  Since
$H$ has abelian Sylow $p$-subgroups, $r<p^2$. Suppose first that $H=\fA_r$.
Since $p$ divides $|H|$, $p\leq r$ and we are done by Lemma~\ref{alt}.
Similarly, if $H=\fS_r$, we deduce that $p=2$.  But then $H$ does not have
abelian Sylow $p$-subgroups, which is a contradiction.
\end{proof}

\subsection{Almost simple groups of Lie type}
We now aim to prove the following, which is slightly stronger than what is
needed later on:

\begin{thm}   \label{thm:almost}
 Let $A$ be almost simple with socle a simple group $S$ of Lie type, and $p$ a
 prime such that $A=O^{p'}(A)$. If $S<A$ and the Sylow $p$-subgroups of $A$ are
 non-abelian then there exists $\chi$ in the principal $p$-block of $A$ of
 height~1.
\end{thm}

That is, the minimal non-zero height in the principal block is as small as
possible. Note that the conclusion fails in general when $A=S$ and $p$ is the
defining characteristic of~$S$ (see \cite{BM15} for examples).
We will give the proof in several steps. 
We will use the following:

\begin{lem}   \label{lem:B0}
 Let $A$ be a finite group and let $S\lhd A$ be a perfect group. Let $p$ be a
 prime and let $\chi\in\Irr(B_0(S))$ with stabiliser $A_\chi$ in $A$. Suppose
 that one of the following holds:
 \begin{enumerate}[\rm(1)]
  \item $\chi(1)$ is not divisible by $p$ and $|A:A_\chi|_p=p$; or
  \item $\chi(1)_p=p$, $|A:A_\chi|_p=1$ and $\chi$ extends to $P$, where
   $P/S\in\Syl_p(A_\chi/S)$.
 \end{enumerate} 
 Then there exists $\psi\in\Irr(B_0(A))$ of height 1.
\end{lem}

\begin{proof}
Suppose that (1) or (2) holds. Then $\chi$ extends to $P$ by \cite[Cor.~6.28]{isa} or the hypothesis, where
$P/S\in\Syl_p(A_\chi/S)$. By \cite[Lemma 4.3]{Mu} there exists
$\vhi\in\Irr(B_0(A_\chi))$ with $\vhi(1)_p=\chi(1)_p$ lying over $\chi$.
By the Clifford correspondence we know that $\psi:=\vhi^A\in\Irr(A)$. By
\cite[Cors~6.2 and~6.7]{N98}, $\psi$ lies in $B_0(A)$ and
\begin{equation*}
\psi(1)_p=|A:A_\chi|_p\chi(1)_p=p.\qedhere\end{equation*}
\end{proof}

Throughout, let $\bG$ be a simple linear algebraic group over an algebraically
closed field~$k$ of characteristic~$r>0$. Let $\bT\le \bG$ be a maximal torus
contained in a Borel subgroup $\bB\le\bG$ and let $\bU$ be the unipotent
radical of~$\bB$. We denote by $\Phi$ the root system of $\bG$ with respect to
$\bT$, and by $\Phi^+\subset\Phi$ the set of positive roots with respect to
$\bB$. Write $\Delta\subset\Phi^+$ for the set of simple roots. For any
$\al\in\Phi^+$, we denote by $\bU_\al$ the corresponding root subgroup in $\bU$
normalized by $\bT$, and we choose an isomorphism
$x_\al:k\rightarrow \bU_\al$ (see e.g.~\cite[Sect.~8]{MT}).   \par
For any graph automorphism $\gamma$ of (the Dynkin diagram of) $\Phi$ and for
any power $q=r^f$, there is a Frobenius endomorphism $F:\bG\to\bG$ with the
property that $F(x_\al(c))=x_{\ga(\al)}(c^q)$ for all $\al\in\Phi$,
$c\in k$ (see \cite[Thm~1.15.2]{GLS}). We set $F_r:\bG\to\bG$,
$x_\al(c)\mapsto x_\al(c^r)$. Note that $F,F_r$ commute in their action on
$\bG$, and $\bT,\bB$ and $\bU$ are $F$- and $F_r$-stable. Furthermore $F$ acts
on the set of subgroups $\bU_\al$, which induces an action of $F$ on the space
$V=\RR\Phi$. The resulting map is $q\phi$, where $\phi$ is an automorphism of
$V$ of finite order. We recall the setup from \cite[\S23]{MT}. Write
$\pi:V\rightarrow V^{\phi}$ for the projection (Reynolds operator) onto the
fixed space $V^{\phi}$.
We define an equivalence relation $\sim$ on $\pi(\Phi)$ by setting
$\pi(\al)\sim\pi(\be)$ if and only if there is some positive $c\in\RR$ such
that $\pi(\al)=c\pi(\be)$, and we let $\widehat{\Phi}$ be the set of equivalence
classes under this relation. For $\al\in\Phi^+$, we write $\wal$ for the class
of $\pi(\al)$ and set $\bU_\wal=\prod_{\be:\pi(\be)\sim\pi(\al)}\bU_\be$. Then
$U_\wal=\bU^F_\wal$ is the corresponding root subgroup of $G:=\bG^F$.
The different possible root subgroups are described in~\cite[Table~2.4]{GLS},
for example. Now, by~\cite[Thm~2.3.7]{GLS},
$U=\bU^F=\prod_{\{\wal\mid\al\in\Phi^+\}}U_\wal$ is a Sylow $r$-subgroup of~$G$.

We first consider the case when $p$ is the defining prime for $S$. For this let
$\bG^*$ be dual to $\bG$ with corresponding Frobenius endomorphisms again
denoted $F,F_r$ and $G^*:=\bG^{*F}$.

\begin{lem}   \label{lem:s in dual}
 In the notation introduced above, let $A\le\Aut(G)$ and suppose $r=p$. Assume
 there is a semisimple element $s\in [G^*,G^*]$ such that the stabiliser of the
 $G^*$-conjugacy class of $s$ has index $p$ in $A$. Then $B_0(A)$ contains a
 character of height~1.
\end{lem}

\begin{proof}
Let $\chi\in\cE(G,s)$ be a semisimple character in the Lusztig series labelled
by~$s$. Then $\chi(1)$ is prime to $p$ by the degree formula
\cite[Thm~2.6.11 with Cor.~2.6.18]{GM20}. Furthermore, $\chi$ is trivial on
$Z(G)$ since $s\in[G^*,G^*]$ (see \cite[Lemma~4.4(ii)]{NT13}). By the main
result of \cite{Hum}, then $\chi\in\Irr(B_0(G))$ and thus $\chi$ can be
considered as a character in $\Irr(B_0(S))$ by \cite[Thms~9.9(c)
and~9.10]{N98}. 
Finally, by \cite[Prop.~7.2]{Tay} our assumption on the class of $s$ implies
that the stabiliser of $\chi$ in $A$ has index~$p$, so by Clifford theory and
Lemma~\ref{lem:B0} there is a character of height~1 above $\chi$ in
$\Irr(B_0(A))$.
\end{proof}

\begin{prop}   \label{prop:field}
 Let $S$ be a simple group of Lie type in characteristic~$p$, $\sigma$ a
 non-trivial $p$-power order field automorphism of $S$ and
 $A:=\langle S,\sigma\rangle\le\Aut(S)$. Let $P\in\Syl_p(A)$. Then:
 \begin{enumerate}[\rm(a)]
  \item there exists $\chi\in\Irr(B_0(A))$ of height~$1$; and
  \item there exists $\theta\in\Irr(P)$ of height~$1$.
 \end{enumerate}
 In particular, the EM-conjecture holds for the principal $p$-block of $A$.
\end{prop}

\begin{proof}
The assumptions imply that $S$ is not a Suzuki group nor a big Ree group (since
these are defined in characteristic~2 but only have odd order field
automorphisms). Let us postpone the treatment of the small Ree groups for the
moment. Thus, now $S=G/Z(G)$ for $\bG$ a simple algebraic group of simply
connected type as introduced
above, with $r=p$. By definition any field automorphism $\sigma$ of $S$ is
induced, up to inner automorphisms, by a suitable power $F_1:=F_p^e$, and up to
replacing $\sigma$ by a primitive power, we may even assume $e|f$. So let $e|f$
be such that $\sigma=F_1|_G$. In particular, $f/e$ is a $p$-power.

For part (a), as above let $\bG^*$ be dual to $\bG$ and $G^*:=\bG^{*F}$. Let
$d$ be maximal such that the cyclotomic polynomial $\Phi_d$ divides the order
polynomial of $\bG^*$ (see \cite[Tab.~24.1]{MT}) and let $l$ be a primitive
prime divisor of $\Phi_{dp}(p^e)$.
Note that $d\ge2$ by loc.~cit.~and hence such a divisor always exists unless
$p=2$, $d=3$, $e=1$ (see \cite[Thm~28.3]{MT}). In the latter case we instead
choose $l=5$. Then there exists a (semisimple) element
$s\in[\bG^{*F_1^p},\bG^{*F_1^p}]$ of order~$l$. If the $\bG^*$-conjugacy class
$C$ of $s$ were $F_1$-stable, by the Lang--Steinberg theorem
\cite[Thm~21.11]{MT} there would exist $F_1$-stable elements in $C$, which is
not the case since by construction $l$ does not divide the order of
$\bG^{*F_1}$. Now an application of Lemma~\ref{lem:s in dual}(1) shows the
existence of a character of $A$ above $\chi$ in $B_0(A)$ with height~1,
proving~(a).   \par
For~(b) let $U=\bU^F$, considered as a Sylow $p$-subgroup of $S$, which is
possible since $Z(G)$ has order prime to $p$. Now by~\cite[Lemma~7]{Hw73}, we
have
$[U,U]=\prod_{\al\in\Phi^+\setminus\Delta}U_\wal$ unless
$$G\in\{B_n(2),\,C_n(2),\,G_2(3),\,F_4(2),\,\tw2B_2(2),\,{}^2G_2(3),\,
  \tw2F_4(2)\}.$$
Since none of the excluded groups does possess non-trivial field automorphisms
we may in fact assume $[U,U]=\prod_{\al\in\Phi^+\setminus\Delta}U_\wal$. Note
that $U$ is $\sigma$-stable. Now
any $c\in k^{F_1^p}\setminus k^{F_1}$ lies in an $F_1$-orbit of
length~$p$. So for any $\al\in\Delta$, the element
$\prod_{\al\in\wal}x_\al(c)[U,U]\in U_\wal[U,U]$ lies in a $\sigma$-orbit of
length~$p$ in $U/[U,U]$. Thus, by Brauer's permutation lemma, there also is a
$\sigma$-orbit of length~$p$ on $\Irr(U/[U,U])$, hence on $\Irr(U)$. That is,
there is a linear character of $U$ whose inertia group in
$\langle U,\sigma\rangle\in\Syl_p(A)$ has index~$p$. So, by Clifford theory,
$P$ has an irreducible character of degree~$p$, showing (b).   \par
Finally, if $S=G={}^2G_2(q^2)$ is a small Ree group and $p=3$, then for any
semisimple element $s\in G^{\sigma^3}\setminus G^\sigma$ the corresponding
semisimple character in $\cE(G,s)$ is as in Lemma~\ref{lem:s in dual}(1),
giving~(a). Further, using that the
description in \cite[\S23]{MT} also applies in this very twisted case, as before
we can construct a linear character $\chi$ of $P\cap S$ stable under $\sigma^3$
but not under~$\sigma$, and then any character of $P$ above $\chi$ is as claimed
in~(b).
\end{proof}

\begin{prop}   \label{prop:defchar}
 Theorem~{\rm\ref{thm:almost}} holds for $A$ with socle a simple group of Lie
 type in characteristic~$p$.
\end{prop}

\begin{proof}
The outer automorphisms of $p$-power order of simple groups of Lie type $S$ in
characteristic~$p$ are field, graph and graph-field automorphisms (see e.g.
\cite[Thm~24.24]{MT}). In particular, for $p\ge5$ only field automorphisms
occur, and these are central in $\Out(S)$. Hence in this case, our claim
follows from Proposition~\ref{prop:field}.   \par
If $p=3$ the only groups $S$ with additional $p$-automorphisms are those of
untwisted type $D_4$, for which $\Out(S)\cong\fS_4\times C_f$.
Thus $A/S\le \fA_4\times C_{3^k}$ with $3^k|f$. By Proposition~\ref{prop:field}
we may assume that $A$ induces some non-field automorphisms of 3-power order.
Let $\bG$ be simply connected and $\bG^*$, $F$, $G=\bG^F$ be as above, so that
$S=G/Z(G)$. Let $3^e$ be
the exponent of a Sylow 3-subgroup of $A/S$. Let $s\in\bG^{*F}$ be a
(semisimple) element whose order is a primitive prime divisor $l$ of
$\Phi_4(3^{f/3^e})$ (so that in fact $s\in[\bG^{*F},\bG^{*F}]$). Then $l$ does
not divide the order of $G_2(3^f)$, the centraliser of a graph automorphism
of~$S$, and thus the $G^*$-conjugacy class of $s$ is not stabilised by the
graph automorphism of $S$. On the other hand, by construction $s$ is fixed by
the field automorphisms in $A/S$. Hence again we may conclude with
Lemma~\ref{lem:s in dual}.   \par
Finally, assume $p=2$. The only groups with graph automorphisms of order two
are those of types $A_n$ ($n\ge2$), $D_n$ ($n\ge4$) and $E_6$.
By the above we may assume $A$ induces some non-field automorphisms of 2-power
order. Let $2^e$ be the exponent of a Sylow 2-subgroup of $A/S$ and $\bG$, $F$
and $\bG^*$ as before with $G:=\bG^F=E_6(2^f)$ and set $F_1:=F_p^{f/2^e}$. Let
$s\in\bG^{*F_1}$ be a (semisimple) element whose order is a primitive
prime divisor of $\Phi_9(2^{f/2^e})$. Then $s$ is stable under all 2-power
order field automorphisms in $A$ but inverted by a graph automorphism (while
it is non-real in $G^*$). Thus the class of $s$ lies in an orbit of length~2
under $A$ and we may apply Lemma~\ref{lem:s in dual}. Next, let $G=\SL_n(2^f)$
with $n\ge3$ 
and $s\in\bG^{*F_1}$ be a (semisimple) element whose order is a primitive
prime divisor of $\Phi_n(2^{f/2^e})$ if $n$ is odd, respectively of
$\Phi_{n-1}(2^{f/2^e})$ if $n$ is even. Then $s$ is non-real, but inverted by
a graph automorphism and we may conclude as in the previous case.

Next, assume $G=D_n(q)$, with $q=2^f$ and $n\ge4$. Here, $\bG^{*F}\cong
\SO_{2n}^+(q)$. We now argue as in the proof of \cite[Prop.~2.8]{MN23}.
Let $s\in\bG^{*F_1}$ be a (semisimple) element
in the stabiliser $\GL_n(q)$ of a maximal totally isotropic subspace of
$\bG^{*F}$ whose order is a primitive prime divisor of $\Phi_n(2^{f/2^e})$.
If $n$ is odd, $s$ is non-real but the graph automorphism of $S$ induces the
transpose-inverse automorphism of $\GL_n(q)$, so conjugates $s$ to its inverse.
If $n$ is even, then $s$ is real but the graph automorphism interchanges the two
conjugacy classes of stabilisers of totally singular subspaces, hence does not
fix the class of~$s$. So in either case the semisimple character $\chi$ of $S$
labelled by~$s$ is invariant under $F_1$ but not under the graph automorphism,
so gives rise to a character in $\Irr(B_0(A))$ of height~1.
\par
Finally, let $S=\Sp_4(2^f)$ or $S=F_4(2^f)$ with $A/S$ inducing (and hence
being generated by) a graph-field automorphism $\sigma$. Set $a=f_{2'}$ and let
$s$ be a regular semisimple element in the first factor of an
$A_1(2^a)\times\tilde A_1(2^a)$-, respectively
$A_2(2^a)\times\tilde A_2(2^a)$-subsystem subfield subgroup. Then $s$ is
centralised by 2-power order field automorphisms, but not by $\sigma$, since
$s$ centralises a long root element, but no short root element, and these two
are interchanged by $\sigma$. Then the semisimple character $\chi_s$ is
again as desired.
\end{proof}

Next, we consider groups of Lie type in characteristic $r$ different from $p$.

\begin{prop}   \label{prop:crosschar-non}
 Theorem~{\rm\ref{thm:almost}} holds for $A$ with socle $S$ simple of Lie type
 in characteristic~$r\ne p$ with non-abelian Sylow $p$-subgroups.
\end{prop}

\begin{proof}
Let $G=\bG^F$ with $\bG$ simple of simply connected type such that $S=G/Z(G)$.
For $p=2$, \cite[Prop.~4.5]{BM15} shows the existence of a unipotent character
of $S$ of $2$-height~1 in $B_0(G)$, except for $A_1(q)$ and $A_2(\eps q)$ with
$q\equiv-\eps\pmod4$. For $p\ge3$ the existence of a unipotent character of
height~1 in $B_0(G)$ is shown in \cite[Prop.~4.3 and Thm~4.7]{BM15}, except for
$G_2(q)$, $\tw3D_4(q)$, $\tw2F_4(q^2)$, $A_2(\eps q)$ and $A_5(\eps q)$ with
$q\equiv\eps\pmod3$ when $p=3$. Since unipotent characters have $Z(G)$ in their
kernel, all of the above can be regarded as characters in $B_0(S)$.
Furthermore, using \cite[Thm~4.5.11]{GM20} one sees that these unipotent
characters are invariant under $p$-automorphisms of $S$ and since by
\cite[Thm~2.4]{Ma08} they extend to their inertia group in $A$, this provides
the desired character of $A$ by Clifford theory and Lemma~\ref{lem:B0}.
\par
It remains to discuss the types left open above. First assume $p=3$.
For $S=G=G_2(r^f)$ let $s\in H:=\bG^{*F_r}=G_2(r)$ be a 3-element in a maximal
torus $T$ of $H$ of order $(r-\eps)^2$, where $\eps\in\{\pm1\}$ and
$r\equiv\eps\pmod3$, whose centraliser in a Sylow 3-subgroup of $H$ has index~3
(since $N_H(T)$ contains a Sylow 3-subgroup of $H$ and the latter is
non-abelian). Then $s$ is invariant under
all field automorphisms of $S$, and any semisimple character in $\cE(G,s)$ has
height~1 (by the degree formula \cite[Thm~2.6.11 with Cor.~2.6.18]{GM20}), lies
in the principal 3-block of $G$ by \cite[Thm~B]{En00} and is invariant under
field automorphisms (since $s$ is), so extends to $A$ as $A/S$ is cyclic.
Exactly the same construction allows one to deal with all of the other excluded
types for $p=3$ if $A$ induces only field automorphisms.   \par
So next let $S=\PSL_3(q)$ with $q\equiv1\pmod3$ and assume $A$ induces some
non-field automorphisms. Let $3^e$ be the exponent of a Sylow 3-subgroup of
$A/S$ and set $F_1:=F_r^{f/3^e}$.
Let now $\bG$ be simple of adjoint type such that $G:=\bG^F=\PGL_3(q)$, and
so $S=[G,G]$. Let $s=\diag(1,\om,\om^2)\in\bG^{*F}=\SL_3(q)$ with
$\om\in\FF_q^\times$ a primitive third root of unity. Thus $s$ is a semisimple
element of order three. Since $r^{f/3^e}\equiv r^f=q\equiv1\pmod3$ this is
$F_1$-stable. Let $\chi\in\cE(G,s)$ be the semisimple character. Then $\chi$
lies in $B_0(G)$ since $G$ has a unique unipotent 3-block. Furthermore, $\chi$
has height~1, while its restriction to $[G,G]=S$ splits into three characters
$\chi_1,\chi_2,\chi_3$ of height~0, since the image of $s$ in
$\bG^*/Z(\bG^*)=\PGL_3$ has disconnected centraliser. Since $s$ is
$F_1$-invariant, the inertia group of $\chi_1$ in $A$ has index~3, yielding the
claimed character of $A$ by Lemma~\ref{lem:s in dual}(1). For $S=\PSL_6(q)$ let
$s=\diag(1,1,1,1,\om,\om^2)\in\SL_6(q)$, an $F_1$-invariant
semisimple element of order three. Let $\chi\in\cE(\PGL_6(q),s)$ be the
corresponding semisimple character. Then $\chi$ lies in the principal 3-block,
since there is a unique unipotent block, has height~1 and is invariant under
field automorphisms. So we may conclude as in the previous case. The unitary
groups $\PSU_3(q)$ and $\PSU_6(q)$ are handled entirely similarly.
\par
Finally let $p=2$ and $S$ one of the groups excluded in the first paragraph.
If $S=\PSL_3(\eps q)$ with $q\equiv-\eps\pmod4$ let $s$ be an element of
order~4 in a maximal torus of $\PGL_3(\eps q)$ of order $q^2-1$. The
corresponding semisimple character of $S$ has degree $q^3-\eps$, hence
height~1, lies in the principal 2-block and is $A$-invariant as $s$ is
centralised by the field and graph automorphisms. Now note that $A/S$ is cyclic
since $q$ is not a square when $\eps=1$, so $\chi$ extends to a character
of~$A$ as required.
\par
Now let $S=\PSL_2(q)$. Since by assumption Sylow 2-subgroups of $S$ are
non-abelian we have
$q\equiv\eps\pmod8$ with $\eps\in\{\pm1\}$. Let $s\in\SL_2(q)$ be an element of
order~4 whose image in $\PGL_2$ has disconnected centraliser. The (semisimple)
character $\chi\in\cE(\PGL_2(q),s)$ has degree~$q+\eps$, hence height~1, lies in
the principal 2-block (being labelled by a 2-element) and is invariant under
the field automorphisms of $S$. Thus we are done whenever $A$ induces some
non-field automorphism. If $A$ induces only field automorphisms, then $q$ is
a square, so there exists a (rational) element of order~4 in a maximal torus
of order $q-1$ of $\PGL_2(q)$ which is a square, so lies in $\PSL_2(q)$. So
the corresponding semisimple character of $\SL_2(q)$ of degree $q+1$ has
$Z(\SL_2(q))$ in its kernel, lies in the principal 2-block, is of height~1
and is invariant under all field automorphisms. It thus extends to a character
of $A$ as desired.
\end{proof}

\begin{prop}   \label{prop:crosschar-ab}
 Theorem~{\rm\ref{thm:almost}} holds for $A$ with socle $S$ simple of Lie type
 in characteristic~$r\ne p$ with abelian Sylow $p$-subgroups.
\end{prop}

\begin{proof}
Note that $S={}^2G_2(q^2)$ has no even order outer automorphisms. Thus, if $p=2$
then Sylow 2-subgroups of $S$ being abelian forces $S=\PSL_2(q)$ with
$q\equiv\eps3\pmod8$, $\eps\in\{\pm1\}$. In particular, $S$ has no 2-power order
field automorphisms and hence $A=\PGL_2(q)$. Here we take the irreducible
character $\chi$ of $A$ labelled by an element of order~4, of degree~$q+\eps$,
which has height~1 and lies in the principal 2-block. So we have $p>2$.
\par
Let $p=3$. For the groups $S=\PSL_3(\eps q)$ with abelian Sylow 3-subgroup and
outer automorphisms of order~3 the arguments in the proof of
Proposition~\ref{prop:crosschar-non} still do apply. The only other cases
with abelian Sylow 3-subgroups are when $S=\PSL_2(q)$, with $q$ not a 3-power
by Proposition~\ref{prop:defchar}. If $A$ has non-abelian Sylow 3-subgroups
then $q=r^f$ with $3|f$. Write $|A/S|=e$ and let $s$ be a 3-element in
$\PGL_2(r^{3f/e})$ not lying in $\PGL_2(r^{f/e})$. Then the semisimple
character $\chi$ of $S$ labelled by~$s$ lies in an orbit of length~3 under $A$,
in the principal 3-block, and has height~0; hence there is a character
of~$A$ above $\chi$ as claimed.

For $p\ge5$, outer automorphisms of order $p$ are either field automorphisms
or $S=\PSL_n(\eps q)$ with $q\equiv\eps\pmod p$, but in the latter case, Sylow
$p$-subgroups of $S$ are non-abelian. Thus we have that $A$ induces $p$-power
field automorphisms of $S$, say generated by a Frobenius map $F_1$. Let $\bG$
be of simply connected type such that $S=G/Z(G)$ for $G=\bG^F$. Let $s\in G^*$
be a $p$-element whose class is invariant under $F_1^p$ but not under $F_1$ and
$\chi\in\cE(G,s)$ the corresponding semisimple character. Since Sylow
$p$-subgroups of $S$ and hence of $G$ are abelian, this has height~0. Then
$\chi\in\Irr(B_0(G))$ by \cite[Thm~B]{En00}. Note that $p$ does not divide
$|G^*:[G^*,G^*]|$, so $\chi$ is a character of $S$. By construction its inertia
group in $A$ has index~$p$, and thus we are done by Lemma~\ref{lem:B0}.
\end{proof}

This completes the proof of Theorem~\ref{thm:almost}. Given a group $G$ we
write $\cd(G)$ to denote the set of character degrees of $G$.

\begin{cor}   \label{cor:as}
 Let $p$ be a prime and let $A$ be an almost simple group with socle $S$ and
 $\bO^{p'}(A)=A$. Let $P\in\Syl_p(A)$ and assume that $\cd(P)=\{1,p^a\}$ for
 some $a\ge1$. Then there exists $\chi\in\Irr(B_0(A))$ such that
 $1<\chi(1)_p\leq p^a$.
\end{cor}

\begin{proof}
The EM-conjecture for principal blocks of simple groups was proved
in \cite[Thm~4.7]{BM15}, so we may assume $A$ is not simple. If $S$ is sporadic
then $|A:S|=2$, so $p=2$, and the result can be checked with \cite{gap}. If
$S$ is alternating then the claim was shown in \cite[Thm~2.1]{BM15}.
Finally, for $S$ simple of Lie type the result is contained in
Theorem~\ref{thm:almost}.
\end{proof}

\section{Proof of Theorem 1}   \label{sec:main}

We will use several times the description of the structure of groups with an
abelian Sylow $p$-subgroup.

\begin{thm}   \label{abel}
 Let $p$ be a prime. Let $G$ be a finite group with $G=\bO^{p'}(G)$ and abelian
 Sylow $p$-subgroups. If $N=\bO_{p'}(G)$, then $G/N=X/N\times Y/N$, where $X/N$
 is an abelian $p$-group and $Y/N$ is trivial or a direct product of non-abelian
 simple groups of order divisible by $p$. In particular, any non-abelian chief
 factor of order divisible by $p$ of a group with abelian Sylow $p$-subgroups
 is a simple group.
\end{thm}

\begin{proof}
This follows from \cite[Thm~2.1]{NT13}.
\end{proof}

We need the following general lemma on principal blocks. It should be compared
with \cite[Lemmas~1.2 and~1.3]{rsv21}.

\begin{lem}   \label{quo}
 Let $G$ be a finite group and $N\unlhd G$. Let $Q\in\Syl_p(N)$ and suppose
 that $\bC_G(Q)\subseteq N$. Then $B_0(G)$ is the unique block of $G$ that
 covers $B_0(N)$. In particular $\Irr(G/N)\subseteq\Irr(B_0(G))$.
\end{lem}

\begin{proof}
Let $B$ be a $p$-block of $G$ covering $B_0(N)$. By Kn\"orr's theorem
\cite[Thm~9.26]{N98}, there exists a defect group $P$ of $B$ containing $Q$.
Since $\bC_G(P)\subseteq\bC_G(Q)\subseteq N$, \cite[Lemma~9.20]{N98} implies
that $B$ is regular with respect to $N$. Now, by \cite[Thm~9.19]{N98},
$B_0(N)^G$ is defined and $B_0(N)^G=B$.

Now, Brauer's Third Main Theorem \cite[Thm~6.7]{N98} implies that $B=B_0(G)$,
as desired. The second assertion follows from \cite[Thm~9.2]{N98}.
\end{proof}

In Step 1 of the proof of Theorem 1 we handle the solvable case. When $p$ is
odd, that case follows from \cite[Thm 6.4 and Prop.~6.5]{EM14}. We extend this
result to the prime~$2$.  We will need orbit theorems of Espuelas and Isaacs
that we recall next. While Espuelas' result shows the existence of regular
orbits in a certain situation, Isaacs' result provides short orbits in a
similar situation.

\begin{thm}   \label{espm}
 Let $p$ be a prime and let $G$ be a solvable group with $\bO_p(G)=1$. Let $A$
 be an abelian $p$-subgroup of $G$. Suppose that $V$ is a faithful $G$-module
 in characteristic $p$. Then there exists $v\in V$ in a regular $A$-orbit.
\end{thm}

\begin{proof}
This is a consequence of the main result of \cite{esp}.
\end{proof}

\begin{thm}   \label{isam}
 Let $p$ be a prime and let $G$ be a solvable group of order divisible by $p$
 with $\bO_p(G)=1$. Let $V$ be a faithful completely reducible $G$-module in
 characteristic $p$. Suppose that $P\in\Syl_p(G)$ is abelian. Then there exists
 $v\in V$ in a $P$-orbit of size $p$.
\end{thm}

\begin{proof}
This is Isaacs' Lemma 5.1 of \cite{EM14}. We take this opportunity to point out
that in that statement it should say `reducible' instead of `irreducible'.
\end{proof}

Finally we recall a particular case of \cite[Problem 6.18]{isa} that will be
used several times in this paper. The result is well-known, but for the
reader's convenience we sketch a proof.

\begin{lem}   \label{sem}
 Let $G=HN$, where $H\le G$, $N\unlhd G$ and $N\cap H=1$. If $\la\in\Irr(N)$
 is linear and $G$-invariant then $\la$ is extendible to $G$. 
\end{lem}

\begin{proof} 
It can easily be checked that $\tilde{\la}:G\longrightarrow\CC^\times$ defined
by $\tilde{\la}(hn)=\la(n)$ for $h\in H, n\in N$, is a group homomorphism,
extending the character $\la$.
\end{proof}

Now, we proceed to prove Theorem 1 which we restate:

\begin{thm}   \label{thm:main}
 Let $G$ be a finite group and $p$ a prime number. Suppose that the set of
 character degrees of a Sylow $p$-subgroup of $G$ is $\{1,p^a\}$, for some
 $a\ge1$. Then there exists $\chi\in\Irr(B_0(G))$ with
 $1\neq \chi(1)_p\leq p^a$.
\end{thm}

\begin{proof}
We proceed by induction on $|G|$.
\medskip

\textit{Step $0$. We may assume $\bO^{p'}(G)=G$ and $\bO_{p'}(G)=1$. Moreover,
 if $1<N\lhd G$, $G/N$ has abelian Sylow $p$-subgroups. Also, there is just one
 minimal normal subgroup of $G$.}
\medskip

Let $P\in \Syl_p(G)$. Let $M=\bO^{p'}(G)$, so that $P\subseteq M$. If $M<G$,
by induction there exists $\psi\in\Irr(B_0(M))$ with $1\neq \psi(1)_p\leq p^a$.
Let $\chi\in\Irr(B_0(G))$ lying over $\psi$, which exists by
\cite[Thm~9.4]{N98}. Then $\chi(1)_p=\psi(1)_p$, and we are done.
\medskip

Now let $N=\bO_{p'}(G)$ and notice that $\Irr(B_0(G/N))=\Irr(B_0(G))$ by
\cite[Thm~9.9(c)]{N98}. If $N>1$ then, since $P\cong PN/N$, by induction there
exists $\chi\in\Irr(B_0(G))$ with $1\neq \chi(1)_p\leq p^a$, as
wanted.
\medskip

Next we show that $G/N$ has abelian Sylow $p$-subgroups for every non-trivial
normal subgroup $N$ of $G$. We have that $\cd(PN/N)\subseteq\cd(P)=\{1,p^a\}$.
Then, if $\cd(PN/N)=\cd(P)$, we are done by induction, since
$\Irr(B_0(G/N))\subseteq\Irr(B_0(G))$. Hence $\cd(PN/N)=\{1\}$ and $PN/N$ is
abelian.
\medskip

The last part of the statement follows from the fact that if $N$ and $M$ are
minimal normal subgroups of $G$, then $G/N$ and $G/M$ have abelian Sylow
$p$-subgroups and hence $G$ is isomorphic to a subgroup of $G/N\times G/M$
(consider the projection $G\rightarrow G/N\times G/M$), which has abelian Sylow
$p$-subgroups, a contradiction since $\cd(P)=\{1,p^a\}$ with $a\geq 1$, by
hypothesis. 
\bigskip

\textit{Step $1$. We may assume that $G$ is not solvable.}
\medskip

Suppose that $G$ is solvable and let $F=\bF(G)$. By Step 0, $F=\bO_p(G)$.
Suppose that $F$ is not elementary abelian, then the Frattini subgroup
$\Phi(G)$ is not trivial and by Step $0$ we have $G/\Phi(G)$ has abelian Sylow
$p$-subgroups. By Gasch\"utz's theorem (see \cite[III.4.2 and III.4.5]{hup})
$G/F$ acts faithfully and completely reducibly on $F/\Phi(G)$. It follows that
$F$ is a Sylow $p$-subgroup of $G$ and by Step 0 $G=F$ is a $p$-group and the
claim follows trivially. Hence we may assume that $F$ is elementary abelian.
Using again Gasch\"utz's theorem, we have that $G=FH$ is a semidirect product
with $H$ acting faithfully and completely reducibly on $F$. Arguing as before
we see that $H/\bO_{p'}(H)$ is a $p$-group. By Schur--Zassenhaus we have that
$H=RQ$, where $R=\bO_{p'}(H)$ and $Q\in\Syl_p(H)$ is abelian.

By Theorem \ref{espm}, there exists a regular $Q$-orbit on $\Irr(F)$. Now, by
Theorem \ref{isam}, there exists a $Q$-orbit of size~$p$ on $\Irr(F)$. Let
$P=FQ\in\Syl_p(G)$. Then $P$ has two character degrees and using
Lemma~\ref{sem}, we conclude that $|Q|=p$. Since $P$ is not abelian,
It\^o's theorem \cite[Cor.~12.34]{isa} implies that there exists
$\chi\in\Irr(G)$ of degree divisible by $p$. Since $F$ is a normal abelian
subgroup and $|G:F|_p=p$, it follows that $\chi(1)_p=p$, and we are done.
\bigskip

\textit{Step $2$. Let $N$ be the unique minimal normal subgroup of $G$. If
 $N=S_1\times\cdots\times S_t$ for non-abelian simple groups $S_i$, then we are
 done.}
\medskip

Since $N$ is the unique minimal normal subgroup of $G$ and it is non-abelian, we have that $\bC_G(N)=1$, so $G\leq \Aut(N)$. If $t=1$, we are done by Corollary~\ref{cor:as}, hence we may assume that $t>1$. Write $H=\bigcap_{i=1}^t\bN_G(S_i)$. We claim that $|G:H|_p\leq p^a$.
\medskip

Write $S=S_1$. Let $Q\in\Syl_p(S)$ so that $R=Q\times\cdots\times Q\in\Syl_p(N)$. Note that $Q>1$ by Step 0. Let $P\in\Syl_p(G)$ containing $R$, so $P\cap N=R$ and let $T=P\cap H\in\Syl_p(H)$.
Note that $\bar{G}=G/H$ transitively permutes the set $\Om=\{S_1,S_2,\ldots, S_t\}$, so $P/T$ is a permutation group on $\Om$.
Assume first that $p$ is odd.
By the first part of Lemma \ref{dol}, there exists $\Delta$ non-empty proper subset of $\Om$ in a regular $P/T$-orbit.
Let $\gamma,\delta\in\Irr(Q)$ be two different characters and let $\mu\in\Irr(R)$ be the character that is a product of copies of $\gamma$ in the positions corresponding to $\Delta$ and copies of $\delta$ elsewhere. It follows that $P_{\mu}\leq \Stab_P(\Delta)= T$. Thus if $\vhi\in\Irr(P)$ lies over $\mu$, $$p^a\geq \vhi(1)\geq|P:T|=|G:H|_p,$$ as wanted.
\medskip

Now, assume that $p=2$. By the second part of Lemma~\ref{dol}, $P/T$ has a
regular orbit on $\fP_3(\Om)$. In other words, there exists a partition
$\Om_1,\Om_2,\Om_3,\Om_4$ of $\Om$ (with at least two non-empty parts) such
that no non-identity element in $P/T$ fixes $\Om_i$ for every $i$. Since a
2-group of order greater than~2 has at least four conjugacy classes, we have
that a Sylow $2$-subgroup of $S$ has at least $4$ conjugacy classes. Hence, we
may choose four different characters
$\delta_1,\delta_2,\delta_3,\delta_4\in\Irr(Q)$. Let $\delta$ be the character
of $R$ that is a product of copies of $\delta_i$ in the positions corresponding
to $\Om_i$ for $i=1,2,3,4$ (if $\Omega_i$ is empty then the character
$\delta_i$ does not appear).  As before, we see that
$P_\delta\subseteq T$. Thus if $\vhi\in\Irr(P)$ lies over $\delta$,
$$2^a\geq \vhi(1)\geq|P:T|=|G:H|_2,$$
as wanted. Thus the claim is proven.
\medskip

Next we claim that $\bC_G(R)\subseteq H$. Let $g\in G\setminus H$. Since $G/H$
is a permutation group on $\{S_1,\dots,S_t\}$, without loss of generality we
may assume that $g$ moves $S_1$. Therefore, for any $x\in Q\setminus\{1\}$,
$g$ does not fix $(x,1,\dots,1)\in R$. It follows that $\bC_G(R)\subseteq H$,
as wanted.
\medskip

By Lemma \ref{simples}, there exists $1_S\neq\alpha\in\Irr_{p'}(B_0(S))$,
$X$-invariant for a Sylow $p$-subgroup $X$ of $\Aut(S)$. Recall that $\bar{G}$
transitively permutes $\{S_1,\ldots, S_t\}$. 
Suppose that the action of $\bar{G}$ on $\Om$ is not $p$-concealed. In this
case, there exists $\Gamma\subseteq \Om$ such that $p$ divides
$|\bar{G}:\Stab_{\bar{G}}(\Gamma)|$. Consider now the character
$\theta\in\Irr_{p'}(B_0(N))$ that is a product of copies of $\al$ in the
positions corresponding to the copies of $S$ in $\Gamma$, and copies
of $1_S$ in the positions corresponding to the copies of $S$ in $\Om-\Gamma$.
Then $\theta$ is invariant in a Sylow $p$-subgroup of~$H$, so $|H:H_\theta|$ is
prime to $p$. Then
$$|G:G_\theta|_p\leq |G:H|_p\leq p^a.$$ 
On the other hand, $G_\theta H/H\subseteq \Stab_{\bar{G}}(\Gamma)$, so $p$
divides $|G:G_\theta|$. It suffices to show that there exists
$\chi\in\Irr(B_0(G))$ with $\chi(1)_p=|G:G_\theta|_p$. Now, $\theta$ is of
$p'$-degree and $N$ is perfect, so if $YN/N\in\Syl_p(G_\theta/N)$, $\theta$
extends to $YN$ by \cite[Cor.~6.28]{isa}. Then, by \cite[Lemma~4.3]{Mu}, there
is $\psi\in\Irr_{p'}(B_0(G_\theta))$ lying over $\theta$. Now
$\chi=\psi^G\in\Irr(B_0(G))$ (by the Clifford correspondence and 
\cite[Cor.~6.2 and Thm~6.7]{N98}) and $\chi(1)_p=|G:G_\theta|_p$, as wanted.
\medskip

Hence we may assume that the action of $\bar{G}$ on $\Om$ is $p$-concealed.
Let $\Delta_1,\ldots, \Delta_r$ be a partition of $\Om$ in blocks of maximal
order (in the sense of Lemma~\ref{pconcealed}). By Lemma~\ref{pconcealed} there
exists a proper normal subgroup $L$ of $G$ such that $H\leq L\lhd G$ and the
action of $G/L$ on $\{\Delta_1,\ldots,\Delta_r\}$ is primitive and
$p$-concealed. We claim that
$$\Irr(G/L)\subseteq\Irr(B_0(G)).$$
Recall that $R=Q\times\cdots\times Q$ is a Sylow $p$-subgroup of $N$, where $Q$
is a Sylow $p$-subgroup of $S$.  Let $U\in\Syl_p(L)$ containing $R$, then
$$\bC_G(U)\subseteq\bC_G(R)\subseteq H\subseteq L.$$
The claim follows from Lemma \ref{quo}.
\medskip

Now, by Lemma \ref{pcon}, there exists $\chi\in\Irr(G/L)$ with $\chi(1)_p=p$.
Since this character belongs to the principal $p$-block of $G$, the result
follows.
\bigskip

\textit{Step 3. Let $N$ be the unique minimal normal subgroup of $G$. If $N\subseteq\bZ(G)$, then we are done.}
\medskip

By Step 0 and Step 2, $N$ is an elementary abelian $p$-group. 
Let $H=P\bC_G(P)$, where $P\in\Syl_p(G)$. If $H=G$, $P$ is normal in $G$, a
contradiction to Step~0. Hence we may assume $H<G$, and by induction there is
$\psi\in\Irr(B_0(H))$ with $1\neq \psi(1)_p\leq p^a$. Now, by
\cite[Thm~4.14]{N98} we have that $B_0(H)^G$ is defined and by the Third Main
Theorem~\cite[Thm~6.2]{N98}, $B_0(H)^G=B_0(G)$. Now,
$$\psi^G=\sum_{\chi\in\Irr(B_0(G))}a_\chi\chi+\sum_{\chi\in\Irr(B_1(G))}a_\chi\chi+\cdots + \sum_{\chi\in\Irr(B_k(G))}a_\chi\chi$$
where $B_i(G)$ are $p$-blocks of $G$. By \cite[Cor.~6.4]{N98} we have
$$\Bigg(\sum_{\chi\in\Irr(B_0(G))}a_\chi\chi(1)\Bigg)_p=\psi^G(1)_p\leq p^a.$$
Therefore there exists an irreducible constituent $\chi\in\Irr(B_0(G))$ of
$\psi^G$ with $\chi(1)_p\leq p^a$. We need to show that $\chi(1)_p\neq 1$. Let
$\la\in\Irr(N)$ be an irreducible constituent of $\psi_N$, so $\la$ lies under
$\chi$. Notice that $\la$ is $G$-invariant, since $N\subseteq\bZ(G)$. Now, if
$1=\chi(1)_p$, then $\chi_P$ has a linear constituent, which is an extension
of~$\la$. Since $P/N$ is abelian by Step 0, all characters in $\Irr(P|\la)$ are
extensions of $\la$ by Gallagher's theorem. Let $\tau\in\Irr(P|\la)$
under~$\psi$, then $\psi(1)_p=\tau(1)=1$, a contradiction. Hence
$1\neq \chi(1)_p$ and the step is proved.
\bigskip

\textit{Step $4$. We may assume that $\bF(G)=\bF^*(G)$.}
\medskip

Let $N$ be the unique minimal normal subgroup of $G$. By Step 0 and Step 2,
$N$ is an elementary abelian $p$-group.
\medskip

By Step 0, $F=\bF(G)=\bO_p(G)>1$. Suppose that $E=\bE(G)>1$ and let $Z=\bZ(E)$.
Since $N$ is the unique minimal normal subgroup of $G$, $N\subseteq Z$ (notice
that $Z>1$ since otherwise $\bF^*(G)=\bF(G)\times E$ in contradiction to
Step~0). We claim that $E/Z=S_1/Z\times\cdots\times S_n/Z$, where
$S_i\trianglelefteq G$ for every $i$. Let $W/Z$ be a (non-abelian) minimal
normal subgroup of $G/Z$ contained in $E/Z$. By the Schur--Zassenhaus theorem
and Step~0, we know that $|W/Z|$ is divisible by $p$. Now, by
Theorem~\ref{abel} applied to~$G/Z$, we have that $W/Z$ is simple and the claim
follows.
Write $S=S_1$, so that $S'$ is a quasi-simple normal subgroup of $G$. Using
again that $N$ is the unique minimal normal subgroup of $G$, we have that
$N\subseteq \bZ(S)\cap S'\subseteq \bZ(S')$.
Looking at the Schur multipliers of the simple groups \cite{GLS}, if $p\geq 5$,
we deduce that $\bZ(S')$ has cyclic Sylow $p$-subgroups. We claim that
$\bZ(S')$ has cyclic Sylow $p$-subgroups for $p=2,3$ as well.

Suppose first that $p=2$. Recall that $G/N$ has abelian Sylow $p$-subgroups,
whence the simple group $S'/N$ has abelian Sylow $p$-subgroups. The simple
groups with abelian Sylow $2$-subgroups were classified by J.~H.~Walter
\cite{wal}. They are $\PSL_2(2^n)$ for $n>1$, $\PSL_2(q)$ where $q\equiv3$ or
$5$ (mod 8) and $q>3$, the Ree groups $^2G_2(3^{2n+1})$ and the Janko
group~$J_1$. As before, we can check in \cite{gap} that these groups have
cyclic multiplier. Therefore $\bZ(S')$ has cyclic Sylow $2$-subgroups.
It remains to consider the case $p=3$. 
It can be checked in \cite{gap} that the unique simple group $S$ whose Schur
multiplier has a non-cyclic Sylow $3$-subgroup is $\PSU_4(3)$, but this group
does not have abelian Sylow $3$-subgroups.

In all cases, we conclude that $N$ is cyclic and hence $|N|=p$. Now, the order
of $G/\bC_G(N)$ divides $p-1$. Using Step 0 we conclude that $N$ is central
in~$G$ and we are done by Step 3. Therefore, we may assume that $E=1$, so that 
$\bF(G)=\bF^*(G)$, as desired.
\bigskip

\textit{Step $5$. There is just one $p$-block in $G$.}
\medskip

By Step 0 and Step 4, we have that $G$ is $p$-constrained, so there is just one
$p$-block (see \cite[Cor.~V.3.11]{feit}, for instance). 
\bigskip

Thus, from now on, we want to find $\chi\in\Irr(G)$ of degree divisible by $p$
such that $\chi(1)_p\leq p^a$. Let $N$ be the unique minimal normal subgroup of
$G$, which is an elementary abelian $p$-group.
Recall that by Step 1 we have that $G$ is not solvable.
\bigskip

\textit{Step $6$. Let $\la\in\Irr(N)$. Then there exists $Q\in\Syl_p(G)$ such
 that $\la$ extends to $Q$. In particular $|G:G_\la|$ is not divisible by $p$.}
\medskip

Let $P\in\Syl_p(G)$ and let $\psi\in\Irr(P|\la)$. If $\psi(1)=1$, then
$\psi_N=\la$, as wanted. Hence we may assume that $\psi(1)=p^a$. Since
$\psi^G(1)_p=\psi(1)=p^a$, there exists an irreducible constituent
$\chi\in\Irr(G)$ of $\psi^G$ with $\chi(1)_p\leq p^a$. If $1<\chi(1)_p$, we are
done. Hence we may assume that $\chi(1)_p=1$. In this case, there exists
$\xi\in\Irr(P)$ with $[\chi_P,\xi]\neq 0$ and $\xi(1)=1$. Let $\mu\in\Irr(N)$
under $\xi$, so $\xi_N=\mu$ and $\mu$ lies under $\chi$. Hence $\mu^g=\la$ for
some $g\in G$. Since $\mu$ extends to $P$, $\mu^g=\la$ extends to $P^g$, as
wanted.
\bigskip

\textit{Step $7$. Let $N$ be the unique minimal normal subgroup of $G$ (which
 is an elementary abelian $p$-group). Suppose there exists $V\lhd G$ with
 $V/N\cong\SL_2(p^n)$ or $V/N\cong\PSL_2(p^n)$ with $p^n>3$. Then $p^a\geq |V/N|_p$}. 
\medskip

Notice that since $V/N$ is perfect and $N$ is the unique minimal normal subgroup of $G$, we have that $V$ is perfect. We claim that if $1_N\neq \la\in\Irr(N)$ then $V_{\la}<V$. Indeed, let $\la\in\Irr(N)$ be $V$-invariant. By Step 6, $\la$ extends to some $R\in\Syl_p(V)$. Since $N$ is a $p$-group, $\la$ also extends to every $q$-Sylow subgroup of $V$ for $q\neq p$ by \cite[Cor.~6.28]{isa}. Then $\la$ extends to $V$ by \cite[Cor.~11.31]{isa}.
But this is not possible since $V$ is perfect. The claim follows.
\medskip

Let $1_N\neq \la\in\Irr(N)$, so $\lambda$ extends to some Sylow $p$-subgroup of $G$ by Step 6. Then $\lambda$ extends to some $R\in\Syl_p(V)$, in particular $R\subseteq V_\la$. Since $V_\la < V$, by Lemma \ref{lemasl}, we have
$V_\la\leq\bN_V(R)$. Let $Q\in\Syl_p(V)$ with $Q\neq R$. Again by
Lemma~\ref{lemasl} we have that $Q\cap R=N$. Now,
$$Q_\la\subseteq V_\la\leq \bN_V(R),$$
whence $Q_\la$ is a subgroup of a Sylow $p$-subgroup of $\bN_V(R)$. It follows
that $Q_\la\subseteq R$, so $Q_\la =N$. Therefore $\la^Q\in\Irr(Q)$, so
$$|Q:N|=\la^Q(1)\in\cd(Q).$$ Since $\lambda^Q$ lies under some irreducible character of some Sylow $p$-subgroup of $G$, we have that $|Q:N|\leq p^a$. 
\bigskip

\textit{Step $8$. We may assume that $\bO_{p'}(G/N)>1$.}
\medskip

Recall that by Step 0, $G/N$ has abelian Sylow $p$-subgroups. Suppose
$\bO_{p'}(G/N)=1$. Then, applying Theorem~\ref{abel} to $G/N$, we deduce that
there exist $X,Y\trianglelefteq G$ containing $N$ such that $X/N$ is an abelian
$p$-group, $Y/N$ is a direct product of non-abelian simple groups $V_i/N$ of
order divisible by $p$ and $G/N=X/N\times Y/N$. Furthermore $V_i/N$ is normal
in $G/N$. Notice that $X$ is a $p$-group, so $1<\bZ(X)\lhd G$. Therefore,
$N\subseteq\bZ(X)$. Then $X\subseteq \bC_G(N)$. Moreover, since $V_i/N$ is
simple, $\bC_{V_i}(N)=V_i$ or $\bC_{V_i}(N)=N$. If the first happens for
some~$i$, then $N\subseteq \bZ(V_i)$ and hence, as in Step~4, it is cyclic.
Then $|N|=p$, and we conclude arguing as in Step~4. Therefore $\bC_{G}(N)=X$.
\medskip

Using Step 6, we now apply Theorem~\ref{t1} and Theorem~\ref{t3} to the action
of $G/X\cong Y/N$ on $\Irr(N)$. Suppose first that this action is primitive, so
that Theorem \ref{t1} applies. 
Since $G/X$ is not solvable, we are not in case (2). Suppose now that we are in
case~(3). In subcase~(i), $p=2$. Since $G/X$ has abelian Sylow $p$-subgroups,
$c=5$ and $c=2^r-1$ or $2^r-2$ for some $r$. This is a contradiction.
Since $G/X\cong Y/N$ is a direct product of simple groups, subcase~(ii) does
not occur. In subcases~(iii), (iv) and~(v), we may assume that $G/X=\PSL_2(11), M_{11}$ or $M_{23}$ and $p=3, 3$ or~$2$, respectively. 
In subcase (iii), there exists $\chi\in\Irr(G/X)$ such that $\chi(1)=12$, so we are done by Step 5. Since $M_{23}$ does not have abelian Sylow $2$-subgroups, subcase (v) does not occur. Suppose that $G/X\cong Y/N=M_{11}$ and $p=3$. Furthermore, $|N|=3^5$, so $Y$ is a perfect group of order $|M_{11}|\cdot3^5=2^4\cdot3^7\cdot5\cdot11$. The group $M_{11}$ has an irreducible character of degree $45$, so its $3$-part is $9$. It suffices to see that the largest character degree of a Sylow $3$-subgroup of $Y$ is at least $9$. We can check in \cite{gap} that there are two perfect groups of order $|Y|$. In both of them, the largest character degree of a Sylow $3$-subgroup is $9$. The result follows in this case too.
\medskip

Finally, we may assume that we are in case (1). As mentioned before, these
groups were classified by Hering. We will use the description in
\cite[App.~1]{liebeck}.
We remark that, as was observed in the proof of \cite[Lemma~6.4]{NT13}, for instance, these groups have at most one non-abelian composition factor. Since $Y/N$ is a direct product of non-abelian simple groups, it suffices to study the simple groups with abelian Sylow $p$-subgroups that appear in Hering's theorem. 
First we consider case (A) of that theorem. Since $Y/N$ is a direct product of non-abelian simple groups of order divisible by $p$, the only possibility is $p=2$ and $Y/N=\SL_2(2^n)$, where $|N|=2^{2n}$ and $n\geq2$. By Step 7 we have that $1<|Y/N|_p\leq p^a$. Moreover if $\chi$ is the Steinberg character of $G/X\cong Y/N$ then $$\chi(1)=|Y/N|_p\in\cd(G/X)$$ and we are done.
\medskip

We do not need to consider the extra-special classes described in (B) because
they are not simple.  Finally, we consider the exceptional classes described in
(C) (see \cite[Tab.~11]{liebeck}). If $Y/N\cong \fA_6$ or $Y/N\cong \fA_7$ and
$p=2$, $Y/N$ does not have abelian Sylow $p$-subgroups. The remaining groups are not simple
\medskip

Now, we may assume that the action of $L=Y/N$ on $N$ is imprimitive. 
We apply Theorem \ref{t3} (recall that $L=\bO^{p'}(L)$ by Step 0). Let
$N=N_1\oplus\cdots\oplus N_n$ be an imprimitivity decomposition for $N$.
Therefore, $\bN_L(N_i)$ is transitive on $N_i\setminus\{0\}$ and
$M:=L/\bigcap\bN_L(N_i)$ induces a primitive $p$-concealed subgroup of $\fS_n$.
Note that $M$ is a factor group of $G$. By Lemma~\ref{pcon}, this group has an
irreducible character $\chi$ such that $\chi(1)_p=p$ and the result follows in
this case.
\bigskip

\textit{Step $9$. If $N$ is the unique minimal normal subgroup, we have that $N=\bF(G)$. In particular, $\bC_G(N)=N$.}
\medskip

Let $K/N=\bO_{p'}(G/N)$. By Step 8, we may assume that $K>N$. By the
Schur--Zassenhaus theorem, there exists a $p$-complement $H$ of $K$, so $K=HN$
and $H\cap N=1$. Then by the Frattini argument we have $G=N\bN_G(H)$.
Write $L=\bN_G(H)$. Now, since $N$ is abelian and normal in $G$, $\bN_N(H)$ is
normal in $G=\bN_G(H)=NL$, and so $\bN_N(H) =1$ or $\bN_N(H) =N$. If
$\bN_N(H)=N$, then $H\lhd G$, and we get a contradiction since $\bO_{p'}(G)=1$
by Step~0. Thus, $L\cap N=\bN_N(H) =1$ and $L$ is a complement of $N$ in~$G$.

Let $F=\bF(G)=\bO_p(G)$. We claim that $F=N$. Notice that $N\subseteq\bZ(F)$
since $\bZ(F)>1$. Let $F_1=F\cap L$. Then $F_1\lhd L$ and since $G=NL$, we have
that $F_1\lhd G$. Since $N$ is the unique minimal normal subgroup, this forces
$F_1=1$, so $F=N$ as claimed. Since $N=\bF(G)=\bF^*(G)$ by Step 4, we have
$\bC_G(N)\subseteq N$, as wanted.
\bigskip

\textit{Step $10$. If $G$ is $p$-solvable, then we are done.}
\medskip

We already know that $G$ is not solvable by Step~1, so $p>2$. Let
$P\in\Syl_p(G)$. By Step~9, $G/N$ acts faithfully on $N$. Hence, $G/N$ acts
faithfully on $\Irr(N)$ and by \cite[Thm~3.1]{mn05} we conclude that there
exists $\la\in\Irr(N)$ such that $P_{\la}=N$. Therefore, by Clifford's
correspondence, $\la^P\in\Irr(P)$ and $p^a=|P:N|$ since $P\neq N$. On the other
hand, since $N$ is an abelian normal subgroup of $G$, it follows from It\^o's
theorem (\cite[Thm~6.15]{isa}, for instance) that for any $\chi\in\Irr(G)$,
$\chi(1)$ divides $|G:N|$. In particular,
$$\chi(1)_p\leq|G:N|_p=|P:N|.$$
Hence, it suffices to see that there exists $\chi\in\Irr(G)$ of degree
divisible by~$p$. Since $G$ does not have a normal abelian Sylow $p$-subgroup,
this follows from \cite[Thm~12.33]{isa}.
\bigskip

\textit{Step $11$. Completion of the proof.}
\medskip

Now, we assume that $G$ is not $p$-solvable. 
Let $K/N=\bO_{p'}(G/N)$. By Step~8, we may assume that $K>N$. Again, applying
Theorem~\ref{abel} to $G/N$, we deduce that there exist $X,Y\trianglelefteq G$
containing $N$ such that $X/K$ is an abelian $p$-group, $Y/K$ is a direct
product of non-abelian simple groups $V_i/N\trianglelefteq G/N$ of order
divisible by $p$ and $G/K=X/K\times Y/K$.

As in Step 9, let $H$ be a $p$-complement of $K$ and let $L=\bN_G(H)$. Then
$G=NL$ and $N\cap L=1$. By Step~9, we have that $L$ acts faithfully and
irreducibly on $\Irr(N)$. By Step~6, this action is $p$-exceptional.
\medskip

Suppose first that the action of $L$ on $\Irr(N)$ is primitive. We apply
Theorem~\ref{t1}. In case (2), $G$ is solvable so this case does not occur.
Assume that we are in case (3). In subcases (i) and (v) we can argue as in
Step~8. In subcases (ii), (iii) and (iv) $G/N$ has a normal subgroup
$V/N=\SL_2(5), \PSL_2(11)$ or $M_{11}$ respectively and in all cases $p=3$.
Note that one of the simple direct factors of $Y/K$ is then $\fA_5,\PSL_2(11)$
or $M_{11}$, respectively. Since the first two groups have irreducible
characters whose degree has $3$-part $3$, the result follows in these cases.
Suppose now that $V/N=M_{11}$. In this case, note that $V$ is perfect and
$|N|=3^5$, so we can complete the proof in this case as in Step 8.
\medskip 

Now, we may assume that we are in case (1). So $L$ is transitive on
$\Irr(N)\setminus\{1_N\}$ and $L$ is one of the groups from Hering's theorem.
Again, we use the description in \cite[App.~1]{liebeck}. 
We start with the infinite classes described in (A). Since $G/N$ is not
solvable and has abelian Sylow $p$-subgroups, it follows that $G/N$ has a
normal subgroup $V/N\cong \SL_2(p^n)$, where $|N|=p^{2n}>3$. By Step~7, we have
that $|V/N|_p\leq p^a$. Now, $\PSL_2(p^n)$ is a composition factor of $G/K$.
Write $M/K$ for this composition factor. If $\chi$ is the Steinberg character
of $M/K$ then
$$\chi(1)=|M/K|_p\in\cd(M/K)\subseteq\cd(G/K)\subseteq\cd(G)$$
and we are done since $|M/K|_p= |V/N|_p$.
\medskip
 
Next, we consider the extra-special classes described in (B). The first four
groups in \cite[Tab.~10]{liebeck} are solvable (in fact they all have a normal
subgroup $R/N\cong {\rm Q}_8$ such that $G/R$ is isomorphic to a subgroup of
$\SL_2(3)$, see \cite[Rem.~XII.7.5]{hb}), so they do not occur. The last case
of this table (where $|N|=3^4)$ can be handled with \cite{gap}.
 
Finally, we consider the exceptional classes described in (C). In the cases
where $p^d\in\{11^2,19^2,29^2,59^2\}$ in \cite[Tab.~11]{liebeck}, we have that
$G/N$ (and hence, $G$) is $p$-solvable, so we are done. Since $G/N$ has abelian
Sylow $p$-subgroups, we are left with the first and the last cases of Table~11.
In the last case, $p=3$ and $G/N\cong\SL_2(13)$. Here, there is
$\chi\in\Irr(G/N)$ of degree~6, and we are done. The first case can be handled
with \cite{gap}.
\medskip

Now, assume that the action of $L$ on $N$ is imprimitive. Arguing as in the
last paragraph of Step 8, we find $\chi\in\Irr(L)$ such that $\chi(1)_p=p$.
This completes the proof.
\end{proof}


\end{document}